\documentclass[12pt]{amsart}

\usepackage{amsthm} 
\usepackage{amsmath} 
\usepackage{amssymb} 

\usepackage{microtype} 
\usepackage{pinlabel} 

\usepackage[
pdfborderstyle={},
pdfborder={0 0 0},
pagebackref,
pdftex]{hyperref}

\renewcommand*{\backref}[1]{}
\renewcommand*{\backrefalt}[4]{
  \ifcase #1 %
   [No citations.]%
  \or
   [#2]%
  \else
   [#2]%
  \fi
}

\newcommand{\calB}{\mathcal{B}}
\newcommand{\calC}{\mathcal{C}}

\newcommand{\calM}{\mathcal{M}}

\newcommand{\calX}{\mathcal{X}}
\newcommand{\calY}{\mathcal{Y}}

\newcommand{\Ann}{\mathbb{A}}

\newcommand{\DD}{\mathbb{D}}

\newcommand{\NN}{\mathbb{N}}

\newcommand{\Sph}{\mathbb{S}} 
\newcommand{\TT}{\mathbb{T}}

\newcommand{\ZZ}{\mathbb{Z}}

\renewcommand{\setminus}{{\smallsetminus}}

\newcommand{\rank}{\operatorname{rank}}

\newcommand{\image}{\operatorname{image}}

\newcommand{\st}{\mathbin{\mid}} 
\newcommand{\from}{\colon} 

\newcommand{\isom}{\cong} 

\newcommand{\bdy}{\partial} 

\newcommand{\diam}{\operatorname{diam}} 

\newcommand{\genus}{\operatorname{genus}}

\newcommand{\MCG}{\mathcal{MCG}} 

\newcommand{\Aut}{\operatorname{Aut}} 

\newcommand{\GL}{\operatorname{GL}} 
\newcommand{\PGL}{\operatorname{PGL}} 

\newcommand{\PMF}{\mathcal{PMF}} 
\newcommand{\FL}{\mathcal{FL}} 
\newcommand{\EL}{\mathcal{EL}} 
\newcommand{\PML}{\mathcal{PML}}  
\newcommand{\Teich}{{Teichm\"uller~}}

\newcommand{\refapp}[1]{Appendix~\ref{App:#1}}
\newcommand{\refthm}[1]{Theorem~\ref{Thm:#1}}
\newcommand{\refcor}[1]{Corollary~\ref{Cor:#1}}
\newcommand{\reflem}[1]{Lemma~\ref{Lem:#1}}
\newcommand{\refprop}[1]{Proposition~\ref{Prop:#1}}

\newcommand{\refrem}[1]{Remark~\ref{Rem:#1}}

\newcommand{\reffig}[1]{Figure~\ref{Fig:#1}}
\newcommand{\refdef}[1]{Definition~\ref{Def:#1}}

\newcommand{\refeqn}[1]{Equation~\ref{Eqn:#1}}

\theoremstyle{plain}
\numberwithin{equation}{section} 
\newtheorem{theorem}[equation]{Theorem}
\newtheorem{corollary}[equation]{Corollary}
\newtheorem{lemma}[equation]{Lemma}

\newtheorem{proposition}[equation]{Proposition}

\theoremstyle{definition}
\newtheorem{definition}[equation]{Definition}
\newtheorem{remark}[equation]{Remark}
\newtheorem*{remark*}{Remark}
\newtheorem{claim}[equation]{Claim}
\newtheorem*{claim*}{Claim}

\newtheorem*{question*}{Question}
\newtheorem*{answer*}{Answer}
\newtheorem*{case*}{Case}
\newtheorem*{application*}{Application}

\theoremstyle{definition}

\newcommand{\fakeenv}{} 

\newenvironment{restate}[2]  
{ 
 \renewcommand{\fakeenv}{#2} 
 \theoremstyle{plain} 
 \newtheorem*{\fakeenv}{#1~\ref{#2}} 
 \begin{\fakeenv}
}
{
 \end{\fakeenv}
}

\usepackage{amscd}

\newcommand{\radius}{{\sf r}}

\newcommand{\cobound}{{\sf c}}
\renewcommand{\image}{{\sf c_0}}
\newcommand{\tight}{{\sf c_1}}
\newcommand{\completion}{{\sf c_2}}
\newcommand{\extension}{{\sf c_3}}

\newcommand{\Cobound}{{\sf C}}

\newcommand{\g}{{\sf g}}
\newcommand{\boundary}{{\sf b}}
\newcommand{\quasi}{{\sf q}}
\newcommand{\Quasi}{{\sf Q}}
\newcommand{\stab}{{\sf M}}
\newcommand{\elementary}{{\sf e}}
\newcommand{\distance}{{\sf d}}
\newcommand{\error}{{\sf K}}

\newcommand{\fancy}{{\sf H}}

\newcommand{\gprod}[1]{\langle #1 \rangle}
\newcommand{\base}{\operatorname{base}}
\newcommand{\QI}{\operatorname{QI}}

\begin{document}

\title{Curve complexes are rigid}

\author{Kasra Rafi}
\email{rafi@math.uchicago.edu}
\author{Saul Schleimer}
\email{s.schleimer@warwick.ac.uk}

\thanks{This work is in the public domain.}

\date{\today}

\begin{abstract}  
Any quasi-isometry of the complex of curves is bounded distance from a
simplicial automorphism.  As a consequence, the quasi-isometry type of
the curve complex determines the homeomorphism type of the surface.
\end{abstract}


\maketitle

\section{Introduction}
\label{Sec:Intro}

The {\em curve complex} of a surface was introduced into the study of
\Teich space by Harvey~\cite{Harvey81} as an analogue of the Tits
building of a symmetric space.  Since then the curve complex has
played a key role in many areas of geometric topology such as the
classification of infinite volume hyperbolic three-manifolds, the
study of the cohomology of mapping class groups, the geometry of
\Teich space, and the combinatorics of Heegaard splittings.

Our motivation is the work of Masur and Minsky~\cite{MasurMinsky99,
MasurMinsky00} which focuses on the coarse geometric structure of the
curve complex, the mapping class group, and other combinatorial moduli
spaces.  It is a sign of the richness of low-dimensional topology that
the geometric structure of such objects is not well understood.

Suppose that $S = S_{\g, \boundary}$ is an orientable, connected,
compact surface with genus $\g$ and $\boundary$ boundary components.
Define the {\em complexity} of $S$ to be $\xi(S)= 3 \g -3 +
\boundary$.  Let $\calC(S)$ be the curve complex of $S$.  Our main
theorem is:

\begin{restate}{Theorem}{Thm:Rigid}
Suppose that $\xi(S) \geq 2$. Then every quasi-isometry of $\calC(S)$
is bounded distance from a simplicial automorphism of $\calC(S)$.
\end{restate}

Before discussing the sharpness of \refthm{Rigid} recall the
definition of $\QI(\calX)$.  This is the group of quasi-isometries of
a geodesic metric space $\calX$, modulo an equivalence relation;
quasi-isometries $f$ and $g$ are equivalent if and only if there is a
constant $\distance$ so that for every $x \in \calX$ we have
$d_\calX(f(x), g(x)) \leq \distance$.  Define $\Aut(\calC(S))$ to be
the group of simplicial automorphisms of $\calC(S)$; notice that these
are always isometries.  From \refthm{Rigid} deduce:

\begin{corollary}
\label{Cor:Auto}
Suppose that $\xi(S) \geq 2$.  Then the natural map 
\[
\Aut(\calC(S)) \to \QI(\calC(S))
\] 
is an isomorphism.
\end{corollary}

\begin{proof}
The map is always an injection.  To see this recall Ivanov's
Theorem~\cite{Ivanov02, Korkmaz99, Luo00}: if $\xi(S) \geq 2$ then
every $f \in \Aut(\calC(S))$ is induced by some homeomorphism of $S$,
called $f_S$.  (When $\xi(S) = 2$ every automorphism of $\calC(S)$ is
induced by some homeomorphism of $S_{0,5}$; see~\cite{Luo00}.)
Suppose that $f \in \Aut(\calC(S))$ is not the identity element.  Then
there is some curve $a$ with $f_S(a)$ not isotopic to $a$.  Consider
the action of $f_S$ on $\PML(S)$.  There is a small neighborhood of
$a$ in $\PML(S)$, say $U$, so that $f_S(U) \cap U = \emptyset$.  Since
ending laminations are dense, $f_S$ moves some ending lamination of
$S$.  By Klarreich's Theorem (see \refthm{Klarreich} below), we deduce
that $f$ moves some point of $\bdy \calC(S)$.  Finally, any isometry
of a Gromov hyperbolic space moving a point of the boundary is
nontrivial in the quasi-isometry group.

On the other hand, \refthm{Rigid} implies that the homomorphism
$\Aut(\calC(S)) \to \QI(\calC(S))$ is a surjection.
\end{proof}

Note that \refthm{Rigid} and \refcor{Auto} are sharp. If $S$ is a
sphere, disk, or pair of pants then the complex of curves is empty.
If $S$ is an annulus then, following~\cite{MasurMinsky00}, the complex
$\calC(S)$ is quasi-isometric to $\ZZ$ (see below) and the conclusion
of \refthm{Rigid} does not hold.
If $S$ is a torus, four-holed sphere or once-holed torus then the
curve complex is a copy of the Farey graph.  Thus $\calC(S)$ is
quasi-isometric to $T_\infty$, the countably infinite valence
tree~\cite{BellFujiwara08}.  Hence $\QI(\calC(S))$ is uncountable
while $\Aut(\calC(S)) = \PGL(2, \ZZ)$ is countable.  Thus, for these
surfaces the conclusion of \refthm{Rigid} does not hold.

We now give an application of \refcor{Auto}:


\begin{theorem}
\label{Thm:Homeomorphic}
Suppose that $S$ and $\Sigma$ are surfaces with $\calC(S)$
quasi-isometric to $\calC(\Sigma)$.  Then either 
\begin{itemize}
\item
$S$ and $\Sigma$ are homeomorphic, 
\item
$\{ S, \Sigma \} = \{ S_{0,6}, S_2 \}$, 
\item
$\{ S, \Sigma \} = \{ S_{0,5}, S_{1,2} \}$, 
\item
$\{ S, \Sigma \} \subset \{ S_{0,4}, S_1, S_{1,1} \}$, or
\item
$\{ S, \Sigma \} \subset \{ \Sph, \DD, S_{0,3} \}$.
\end{itemize}
Thus, two curve complexes are quasi-isometric if and only if they
are isomorphic.
\end{theorem}

To prove \refthm{Homeomorphic} we require \refthm{Rigid} and the
following folk theorem:

\begin{restate}{Theorem}{Thm:Ivanov}
Suppose that $S$ and $\Sigma$ are compact, connected, orientable
surfaces with $\MCG(S)$ isomorphic to $\MCG(\Sigma)$.  Then either
\begin{itemize}
\item
$S$ and $\Sigma$ are homeomorphic, 
\item
$\{ S, \Sigma \} = \{ S_1, S_{1,1} \}$, or
\item
$\{ S, \Sigma \} = \{ \Sph, \DD \}$.
\end{itemize}
\end{restate}

Apparently no proof of \refthm{Ivanov} appears in the literature.  In
\refapp{Ivanov} we discuss previous work (\refrem{Ivanov}) and, for
completeness, give a proof of \refthm{Ivanov}.

\begin{proof}[Proof of \refthm{Homeomorphic}]
For brevity, we restrict to the case where $\xi(S)$ and $\xi(\Sigma)$
are at least four.  By \refcor{Auto} the automorphism groups of
$\calC(S)$ and $\calC(\Sigma)$ are isomorphic.  Ivanov's
Theorem~\cite{Ivanov02, Korkmaz99, Luo00} tells us that the simplicial
automorphism group is isomorphic to the mapping class group.  Finally,
it follows from \refthm{Ivanov} that such surfaces are characterized,
up to homeomorphism, by their mapping class groups.
\end{proof}

\subsection*{Outline of the paper}
The proof of \refthm{Rigid} has the following ingredients.  A pair of
ending laminations is {\em cobounded} if the projections of this pair
to any strict subsurface of $S$ are uniformly close to each other in
the complex of curves of that subsurface (see \refdef{Cobounded}).

\begin{restate}{Theorem}{Thm:Cobounded}
Suppose that $\xi(S) \geq 2$ and suppose that $\phi \from \calC(S) \to
\calC(\Sigma)$ is a quasi-isometric embedding.  Then the induced map
on boundaries preserves the coboundedness of ending laminations.
\end{restate}

\refthm{Cobounded} is important in its own right and may have other
applications. For example, it may be helpful in classifying
quasi-isometric embeddings of one curve complex into another.
(See~\cite{RafiSchleimer09}.)  The proof of \refthm{Cobounded} uses
the following theorem in an essential way:

\begin{theorem}[Gabai~\cite{Gabai09}]
\label{Thm:Gabai}
Suppose that $\xi(S) \geq 2$.  Then $\bdy \calC(S)$ is connected. \qed
\end{theorem}

\begin{remark}
Leininger and the second author~\cite{LeiningerSchleimer09} previously
gave a quite different proof of \refthm{Gabai} in the cases where $S$
has genus at least four, or where $S$ has genus at least two and
non-empty boundary.  Note that Gabai's Theorem is sharp; $\bdy
\calC(S)$ is not connected when $S$ is an annulus, torus, once-holed
torus or four-holed sphere.
\end{remark}

Let $\calM(S)$ denote the {\em marking complex} of the surface $S$.
We show that a marking on $S$ can be coarsely described by a pair of
cobounded ending laminations and a curve in
$\calC(S)$. \refthm{Cobounded} implies that a quasi-isometric
embedding of $\calC(S)$ into $\calC(\Sigma)$ induces a map from
$\calM(S)$ to $\calM(\Sigma)$.

\begin{restate}{Theorem}{Thm:LipschitzExtension}
Suppose that $\xi(S) \geq 2$ and $\phi \from \calC(S) \to
\calC(\Sigma)$ is a $\quasi$--quasi-isometric embedding.  Then $\phi$
induces a coarse Lipschitz map $\Phi \from \calM(S) \to \calM(\Sigma)$
so that the diagram
\[
\begin{CD} 
\calM(S) @>\Phi>> \calM(\Sigma)\\ 
 @VVpV               @VV\pi V  \\ 
\calC(S) @>\phi>> \calC(\Sigma) 
\end{CD} 
\]
commutes up to an additive error. Furthermore, if $\phi$ is a
quasi-isometry then so is $\Phi$.
\end{restate}

As the final step of the proof of \refthm{Rigid} we turn to a recent
theorem of Behrstock, Kleiner, Minsky and
Mosher~\cite{BehrstockEtAl08}.  See also~\cite{Hamenstaedt05b}.

\begin{theorem}
\label{Thm:MCGRigid}
Suppose that $\xi(S) \geq 2$ and $S \not = S_{1,2}$.  Then every
quasi-isometry of $\calM(S)$ is bounded distance from the action of a
homeomorphism of $S$. \qed
\end{theorem}

So, if $f \from \calC(S) \to \calC(S)$ is a quasi-isometry then
\refthm{LipschitzExtension} gives a quasi-isometry $F$ of marking
complexes.  This and \refthm{MCGRigid} imply \refthm{Rigid} except
when $S = S_{1,2}$. But the curve complexes $\calC(S_{0,5})$ and
$\calC(S_{1,2})$ are identical.  Therefore to prove \refthm{Rigid} for
$\calC(S_{1,2})$ it suffices to prove it for $\calC(S_{0,5})$.

\subsection*{Acknowledgements}  
This paper was sparked by a question of Slava Matveyev.  
We thank Dan Margalit for useful conversations.

\section{Background}

\subsection*{Hyperbolic spaces}
A geodesic metric space $\calX$ is {\em Gromov hyperbolic} if there is
a {\em hyperbolicity constant} $\delta \geq 0$ so that every triangle
is $\delta$--{\em slim}: for every triple of vertices $x, y, z \in
\calX$ and every triple of geodesics $[x, y], [y, z], [z,x]$ the
$\delta$--neighborhood of $[x, y] \cup [y, z]$ contains $[z,x]$.

Suppose that $(\calX, d_\calX)$ and $(\calY, d_\calY)$ are geodesic
metric spaces and $f \from \calX \to \calY$ is a map.  Then $f$ is
{\em $\quasi$--coarsely Lipschitz} if for all $x, y \in \calX$ we have
\[
d_\calY(x', y') \leq \quasi\,d_\calX(x, y) + \quasi
\] 
where $x' = f(x)$ and $y' = f(y)$.  If, in addition, 
\[
d_\calX(x, y) \leq \quasi\,d_\calY(x', y') + \quasi
\]
then $f$ is a {\em $\quasi$--quasi-isometric embedding}. Two maps $f,
g \from \calX \to \calY$ are $\distance$--{\em close} if for all $x
\in \calX$ we have
\[
d_\calY(f(x), g(x)) \leq \distance.
\] 
If $f \from \calX \to \calY$ and $g \from \calY \to \calX$ are
$\quasi$--coarsely Lipschitz and also $f \circ g$ and $g \circ f$ are
$\quasi$--close to identity maps then $f$ and $g$ are $\quasi$--{\em
quasi-isometries}.

A quasi-isometric embedding of an interval $[s, t] \subset \ZZ$, with
the usual metric, is called a {\em quasi-geodesic}.  In hyperbolic
spaces quasi-geodesics are {\em stable}:

\begin{lemma}
\label{Lem:Stability}
Suppose that $(\calX, d_\calX)$ has hyperbolicity constant $\delta$
and that $f \from [s, t] \to \calX$ is a $\quasi$--quasi-geodesic.
Then there is a constant $\stab_\calX = \stab(\delta, \quasi)$ so that
for any $[p, q] \subset [s, t]$ the image $f([p, q])$ and any geodesic
$[f(p), f(q)]$ have Hausdorff distance at most $\stab_\calX$ in
$\calX$.  \qed
\end{lemma}


See~\cite{Bridson99} for further background on hyperbolic spaces.

\subsection*{Curve Complexes}
Let $S = S_{\g,\boundary}$, as before. Define the vertex set of the curve
complex, $\calC(S)$, to be the set of simple closed curves in $S$ that
are essential and non-peripheral, considered up to isotopy.

When the complexity $\xi(S)$ is at least two, distinct vertices $a, b
\in \calC(S)$ are connected by an edge if they have disjoint
representatives.

When $\xi(S) = 1$ vertices are connected by an edge if there are
representatives with geometric intersection exactly one for the torus
and once-holed torus or exactly two for the four-holed sphere.  This
gives the {\em Farey graph}.  When $S$ is an annulus the vertices are
essential embedded arcs, considered up to isotopy fixing the boundary
pointwise.  Vertices are connected by an edge if there are
representatives with disjoint interiors.

For any vertices $a, b \in \calC(S)$ define the distance $d_S(a,b)$ to
be the minimal number of edges appearing in an edge path between $a$
and $b$. 

\begin{theorem}[Masur-Minsky~\cite{MasurMinsky99}]
\label{Thm:Hyp}
The complex of curves $\calC(S)$ is Gromov hyperbolic.  \qed
\end{theorem}

We use $\delta_S$ to denote the hyperbolicity constant of $\calC(S)$.

\subsection*{Boundary of the curve complex}
Let $\bdy \calC(S)$ be the Gromov boundary of $\calC(S)$. This is the
space of quasi-geodesic rays in $\calC(S)$ modulo equivalence: two
rays are equivalent if and only if their images have bounded
Hausdorff distance.


Recall that $\PML(S)$ is the projectivized space of measured
laminations on $S$.  A measured lamination $\ell$ is {\em filling} if
every component $S \setminus \ell$ is a disk or a boundary-parallel
annulus.  Take $\FL(S) \subset \PML(S)$ to be the set of filling
laminations with the subspace topology.  Define $\EL(S)$, the space of
ending laminations, to be the quotient of $\FL(S)$ obtained by
forgetting the measures.  See~\cite{Kapovich01} for an expansive
discussion of laminations.

\begin{theorem}[Klarreich~\cite{Klarreich99}]
\label{Thm:Klarreich}
There is a mapping class group equivariant homeomorphism between $\bdy
\calC(S)$ and $\EL(S)$.  \qed
\end{theorem}

We define $\overline{\calC(S)} = \calC(S) \cup \bdy \calC(S)$.  

\subsection*{Subsurface projection}
Suppose that $Z \subset S$ is an {\em essential} subsurface: $Z$ is
embedded, every component of $\bdy Z$ is essential in $S$, and $Z$ is
not a boundary-parallel annulus nor a pair of pants.  An essential
subsurface $Z \subset S$ is {\em strict} if $Z$ is not homeomorphic to
$S$.

A lamination $b$ {\em cuts} a subsurface $Z$ if every isotopy
representative of $b$ intersects $Z$.  If $b$ does not cut $Z$ then
$b$ {\em misses} $Z$.

Suppose now that $a, b \in \overline{\calC(S)}$ both cut a strict
subsurface $Z$.  Define the {\em subsurface projection distance}
$d_Z(a, b)$ as follows: isotope $a$ with respect to $\bdy Z$ to
realize the geometric intersection number.  Surger the arcs of $a \cap
Z$ to obtain $\pi_Z(a)$, a finite set of vertices in $\calC(Z)$.
Notice that $\pi_Z(a)$ has uniformly bounded diameter in $\calC(Z)$
independent of $a$, $Z$ or $S$.  Define
\[
d_Z(a,b) = \diam_Z \big( \pi_Z(a) \cup \pi_Z(b) \big).
\]

We now recall the Lipschitz Projection
Lemma~\cite[Lemma~2.3]{MasurMinsky00}:

\begin{lemma}[Masur-Minsky]
\label{Lem:LipschitzProjection}
Suppose that $\{ a_i \}_{i = 0}^N \subset \calC(S)$ is a path where
every vertex cuts $Z \subset S$.  Then $d_Z(a_0, a_N) \leq 2N$.  \qed
\end{lemma}

For geodesics, the much stronger Bounded Geodesic Image Theorem
holds~\cite{MasurMinsky00, Minsky03}:

\begin{theorem}
\label{Thm:Image}
There is a constant $\image = \image(S)$ with the following property.
For any strict subsurface $Z$ and any points $a, b \in
\overline{\calC(S)}$, if every vertex of the geodesic $[a,b]$ cuts $Z$
then $d_Z(a,b) < \image$.  \qed
\end{theorem}

\subsection*{Marking complex}
We now discuss the {\em marking complex}, following Masur and
Minsky~\cite{MasurMinsky00}.  A {\em complete clean marking} $m$ is a
pants decomposition $\base(m)$ of $S$ together with a {\em
transversal} $t_a$ for each element $a \in \base(m)$.  To define
$t_a$, let $X_a$ be the non-pants component of $S \setminus (\base(m)
\setminus \{a\})$.  Then any vertex of $\calC(X_a)$ not equal to $a$
and meeting $a$ minimally may serve as a transversal $t_a$.  Notice
that diameter of $m$ in $\calC(S)$ is at most $2$.

Masur and Minsky also define {\em elementary moves} on markings.  The
set of markings and these moves define the {\em marking complex},
$\calM(S)$: a locally finite graph quasi-isometric to the mapping
class group.  The projection map $p \from \calM(S) \to \calC(S)$,
sending $m$ to any element of $\base(m)$, is coarsely mapping class
group equivariant.  We now record, from~\cite{MasurMinsky00}, the
Elementary Move Projection Lemma:

\begin{lemma}
\label{Lem:ElementaryI}
If $m$ and $m'$ differ by an elementary move then for any essential
subsurface $Z \subseteq S$, we have $d_Z(m, m') \leq 4$. \qed
\end{lemma}

A converse follows from the {\em distance
estimate}~\cite{MasurMinsky00}.

\begin{lemma}
\label{Lem:ElementaryII}
For every constant $\cobound$ there is a bound $\elementary =
\elementary(\cobound, S)$ with the following property.  If $d_Z(m, m')
\leq \cobound$ for every essential subsurface $Z \subseteq S$ then
$d_\calM(m, m') \leq \elementary$. \qed
\end{lemma}

\subsection{Tight geodesics}
The curve complex is locally infinite. Generally, there are infinitely
many geodesics connecting a given pair of points in $\calC(S)$.  In
\cite{MasurMinsky00} the notion of a {\em tight} geodesic is
introduced.  This is a technical hypothesis which provides a certain
kind of local finiteness.  \reflem{Tight} below is the only property
of tight geodesics used in this paper.

\begin{definition} 
\label{Def:Cobounded}
A pair of curves, markings or laminations $a, b$ are {\em
$\cobound$--cobounded} if $d_Z(a, b) \leq \cobound$ for all strict
subsurfaces $Z \subset S$ cut by both $a$ and $b$.
\end{definition}

Minsky shows \cite[Lemma 5.14]{Minsky03} that if $a, b \in
\overline{\calC(S)}$ then there is a tight geodesic $[a,b] \subset
\calC(S)$ connecting them. All geodesics from here on are assumed to
be tight.

\begin{lemma}[Minsky]
\label{Lem:Tight}
There is a constant $\tight = \tight(S)$ with the following property.
Suppose that $(a, b)$ is a $\cobound$--cobounded pair in
$\overline{\calC(S)}$ and $c \in [a, b]$ is a vertex of a tight
geodesic.  Then the pairs $(a, c)$ and $(c, b)$ are $(\cobound +
\tight)$--cobounded.  \qed
\end{lemma}

\section{Extension Lemmas}
We now examine how points of $\calC(S)$ can be connected to infinity.  

\begin{lemma}[Completion]
\label{Lem:Completion}
There is a constant $\completion = \completion(S)$ with the following
property.  Suppose that $b \in \calC(S)$ and $\ell \in
\overline{\calC(S)}$.  Suppose that the pair $(b, \ell)$ is
$\cobound$--cobounded.  Then there is a marking $m$ so that $b \in
\base(m)$ and $(m, \ell)$ are $(\cobound + \completion)$--cobounded.
\qed
\end{lemma}

The existence of the marking $m$ follows from the construction
preceding \cite[Lemma 6.1]{Behrstock06}.

\begin{lemma}[Extension past a point]
\label{Lem:DistanceOne}
Suppose that $a, z \in \calC(S)$ with $a \neq z$.  Then there is a
point $\ell \in \bdy \calC(S)$ so that the vertex $a$ lies in the
one-neighborhood of $[z, \ell]$.
\end{lemma}

\begin{proof}
Let $k \in \bdy \calC(S)$ be any lamination. Let $Y$ be a component of
$S \setminus a$ that meets $z$. Pick any mapping class $\phi$ with
support in $Y$ and with translation distance at least $(2\image + 2)$
in $\calC(Y)$.  We have either
$$
d_Y(z, k) \geq \image \quad\text{or}\quad d_Y(z, \phi(k)) \geq \image.
$$ By \refthm{Image}, at least one of the geodesics $[z, k]$ or $[z,
\phi(k)]$ passes through the one-neighborhood of $a$.
\end{proof}

\begin{proposition}[Extension past a marking]
\label{Prop:Extension}
There is a constant $\extension = \extension(S)$ such that if $m$ is a
marking on $S$, then there are laminations $k$ and $\ell$ such that
the pairs $(k, \ell)$, $(k, m)$ and $(m, \ell)$ are
$\extension$--cobounded and $[k, \ell]$ passes through the
one-neighborhood of $m$.
\end{proposition}

\begin{proof}
There are only finitely many markings up to the action of the mapping
class group.  Fix a class of markings and pick a representative $m$.
We will find a pseudo-Anosov map with stable and unstable laminations
$k$ and $\ell$ such that $[k, \ell]$ passes through the
one-neighborhood of $m$.  This suffices to prove the proposition:
there is a constant $\extension(m)$ large enough so that the pairs
$(k, \ell)$, $(k, m)$ and $(m, \ell)$ are $\extension(m)$--cobounded.
The same constant works for every marking in the orbit $\MCG(S) \cdot
m$, by conjugation.  We can now take $\extension$ to be the maximum of
the $\extension(m)$ as $m$ ranges over the finitely many points of the
quotient $\calM(S)/\MCG(S)$.

So choose any pseudo-Anosov map $\phi'$ with stable and unstable
laminations $k'$ and $\ell'$.  Choose any point $b' \in [k', \ell']$.
We may conjugate $\phi'$ to $\phi$, sending $(k', \ell', b')$ to $(k,
\ell, b)$, so that $b$ is disjoint from some curve $a \in \base(m)$.
This finishes the proof.
\end{proof}

\section{The shell is connected}

Let $\calB(z, \radius)$ be the ball of radius $\radius$ about $z \in
\calC(S)$. The difference of concentric balls is called a {\em shell}.

\begin{proposition}
\label{Prop:ShellConnected}
Suppose that $\xi(S) \geq 2$ and $\distance \geq \max \{ \delta_S, 1
\}$.  Then, for any $\radius \geq 0$, the shell
$$
\calB(z, \radius + 2\distance) \setminus \calB(z, \radius - 1)
$$
is connected.  
\end{proposition}

Below we will only need the corollary that $\calC(S) \setminus
\calB(z, \radius - 1)$ is connected.  However, the shell has other
interesting geometric properties.  We hope to return to this subject
in a future paper.

One difficulty in the proof of \refprop{ShellConnected} lies in
pushing points of the inner boundary into the interior of the shell.
To deal with this we use the fact that $\calC(S)$ has no {\em dead
ends}.

\begin{lemma}
\label{Lem:NoDeadEnds}
Fix vertices $z, a \in \calC(S)$.  Suppose $d_S(z, a) = \radius$.  Then
there is a vertex $a' \in \calC(S)$ with $d_S(a, a') \leq 2$ and
$d_S(z, a') = \radius + 1$.  \qed
\end{lemma}

Note that this implies that any geodesic $[a, a']$ lies outside of
$\calB(z, \radius - 1)$.  For a proof of \reflem{NoDeadEnds}, see
Proposition~3.1 of~\cite{Schleimer06c}.

\begin{proof}[Proof of \refprop{ShellConnected}]
For any $z \in \calC(S)$ and any geodesic or geodesic segement $[a,b]
\subset \overline{\calC(S)}$ define $d_S(z, [a,b]) = \min \{ d_S(z, c)
\st c \in [a,b] \}$.  Define a product on $\overline{\calC(S)}$ by:
\[
\gprod{a,b}_z = \inf \big\{ d_S(z, [a,b]) \big\}
\]
where the infimum ranges over all geodesics $[a,b]$.  For every $k \in
\bdy \calC(S)$ let
\[
U(k) = \{ \ell \in \bdy \calC(S) \st \gprod{k, \ell}_z 
                                 > \radius + 2\distance \}.
\] 
The set $U(k)$ is a neighborhood of $k$ by the definition of the
topology on the boundary~\cite{Gromov87}.  Notice that if $\ell \in
U(k)$ then $k \in U(\ell)$.

Consider the set $V(k)$ of all $\ell \in \bdy \calC(S)$ so that there
is a finite sequence $k = k_0, k_1, \ldots, k_N = \ell$ with $k_{i+1}
\in U(k_i)$ for all $i$.  Now, if $\ell \in V(k)$ then $U(\ell)
\subset V(k)$; thus $V(k)$ is open.  If $\ell$ is a limit point of
$V(k)$ then there is a sequence $\ell_i \in V(k)$ entering every
neighborhood of $\ell$.  So there is some $i$ where $\ell_i \in
U(\ell)$.  Thus $\ell \in U(\ell_i) \subset V(k)$ and we find that
$V(k)$ is closed.  Finally, as $\bdy \calC(S)$ is connected
(\refthm{Gabai}), $V(k) = \bdy \calC(S)$.

Let $a', b'$ be any vertices in the shell $\calB(z, \radius +
2\distance) \setminus \calB(z, \radius - 1)$.  We connect $a'$, via a
path in the shell, to a vertex $a$ so that $d_S(z, a) = \radius +
\distance$.  This is always possible: points far from $z$ may be
pushed inward along geodesics and points near $z$ may be pushed
outward by \reflem{NoDeadEnds}.  Similarly connect $b'$ to $b$.

By \reflem{DistanceOne} there are points $k, \ell \in \bdy \calC(S)$
so that there are geodesic rays $[z, k]$ and $[z, \ell]$ within
distance one of $a$ and $b$ respectively. Connect $k$ to $\ell$ by a
chain of points $\{ k_i \}$ in $V(k)$, as above.  Define $a_i \in [z,
k_i]$ so that $d_S(z, a_i) = \radius + \distance$.  Connect $a$ to
$a_0$ via a path of length at most $2$.

Notice that $d_S(a_i, [k_i, k_{i+1}]) > \distance \geq \delta$.  As
triangles are slim, the vertex $a_i$ is $\delta$--close to $[z,
k_{i+1}]$.  Thus $a_i$ and $a_{i+1}$ may be connected inside of the
shell via a path of length at most $2\delta$.
\end{proof}

\section{Image of a cobounded geodesic is cobounded}

We begin with a simple lemma: 

\begin{lemma}
\label{Lem:FinitelyMany}
For every $\cobound$ and $\radius$ there is a constant $\error$ with
the following property.  Let $[a, b] \subset \calC(S)$ be a geodesic
segment of length $2\radius$ with $(a, b)$ being
$\cobound$--cobounded.  Let $z$ be the midpoint. Then there is a path
$P$ of length at most $\error$ connecting $a$ to $b$ outside of
$\calB(z, \radius - 1)$.
\end{lemma}

\begin{proof}
There are only finitely many such triples $(a, z, b)$, up the action
of the mapping class group.  (This is because there are only finitely
many hierarchies having total length less than a given upper bound;
see~\cite{MasurMinsky00}).  The conclusion now follows from the
connectedness of the shell (\refprop{ShellConnected}).
\end{proof}

Note that any quasi-isometric embedding $\phi \from \calC(S) \to
\calC(\Sigma)$ extends to a one-to-one continuous map from $\bdy
\calC(S)$ to $\bdy \calC(\Sigma)$.

\begin{theorem}
\label{Thm:Cobounded}
There is a function $\fancy \from \NN \to \NN$, depending only on
$\quasi$ and the topology of $S$ and $\Sigma$, with the following
property.  Suppose $(k, \ell)$ is a pair of $\cobound$--cobounded
laminations and $\phi \from \calC(S) \to \calC(\Sigma)$ is a
$\quasi$--quasi-isometric embedding.  Then $\kappa = \phi(k)$ and
$\lambda = \phi(\ell)$ are $\fancy(\cobound)$--cobounded
\end{theorem}


\begin{proof}
For every strict subsurface $\Omega \subset \Sigma$ we must bound
$d_\Omega(\kappa, \lambda)$ from above.  Now, if $d_\Sigma(\bdy
\Omega, [\kappa, \lambda]) \geq 2$ then by the Bounded Geodesic Image
Theorem (\ref{Thm:Image}) we find $d_\Omega(\kappa,\lambda) \leq
\cobound_0 = \cobound_0(\Sigma)$ and we are done.

\begin{figure}[htbp]
\labellist
\small\hair 2pt
\pinlabel {$\calC(S)$} [Bl] at 144.5 149.6 
\pinlabel {$\calC(\Sigma)$} [Bl] at 349.8 149.6 
\pinlabel {$k$} [Br] at 35.1 154.6 
\pinlabel {$\kappa$} [Br] at 237.4 153.6 
\pinlabel {$\ell$} [t] at 85.8 0 
\pinlabel {$\lambda$} [t] at 308.7 1
\pinlabel {$a$} [Br] at 69.8 94.4
\pinlabel {$\alpha$} [r] at 264.5 109.4
\pinlabel {$b$} [tr] at 84.8 39.6 
\pinlabel {$\beta$} [r] at 292.6 43.7
\pinlabel {$z$} [r] at 78.8 68.8
\pinlabel {$\phi(z)$} [t] at 255 62.2
\pinlabel {$\bdy \Omega$} [l] at 296.6 81.3
\endlabellist
\[
\begin{array}{c}
\includegraphics[scale=.8]{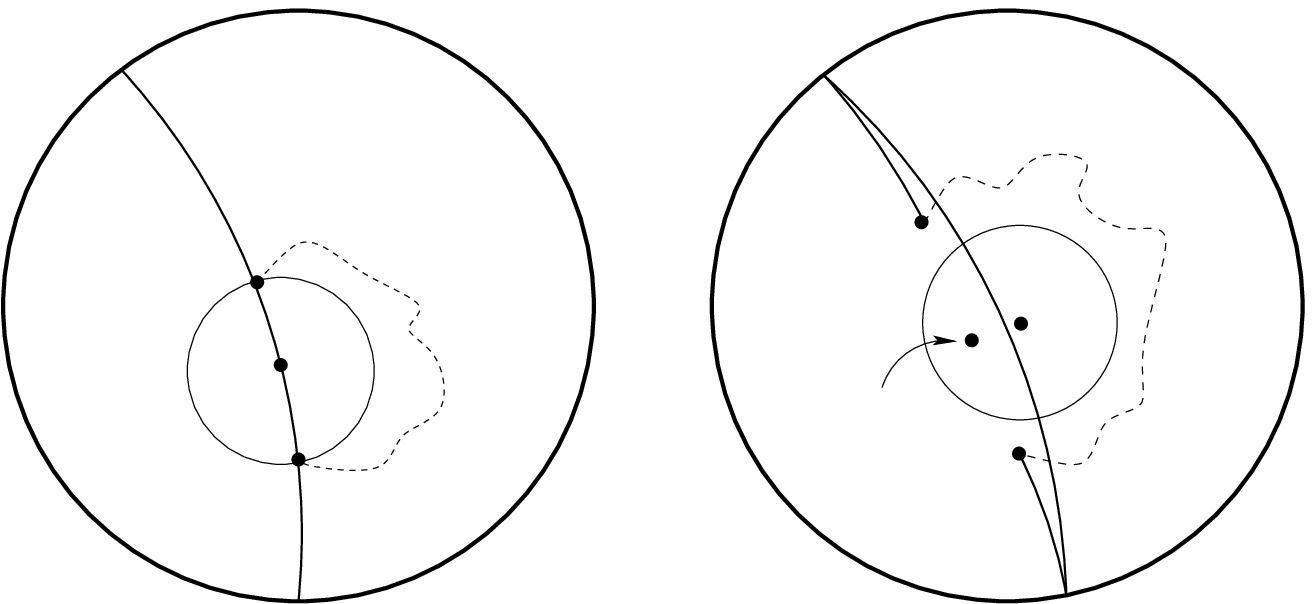}
\end{array}
\]
\caption{Points outside of an $\radius$--ball about $z$ are sent by
  $\phi$ outside of an $(\quasi + \stab_\Sigma + 2)$--ball about $\bdy
  \Omega$.}
\label{Fig:Cobounded}
\end{figure}


Now suppose $d_\Sigma(\bdy \Omega, [\kappa, \lambda]) \leq 1$.  Note
that $[\kappa, \lambda]$ lies in the $\stab$--neighborhood of
$\phi([k, \ell])$, where $\stab = \stab_\Sigma$ is provided by
\reflem{Stability}.  Choose a vertex $z \in [k, \ell]$ so that
$d_\Sigma(\phi(z), \bdy \Omega) \leq \stab + 1$.  Set $\radius =
\quasi(\quasi + 2\stab + 3) + \quasi$.  Thus
\begin{align*}
d_S(y, z)  \geq \radius 
    & \implies d_\Sigma(\phi(y), \phi(z)) \geq \quasi + 2\stab + 3 \\
    & \implies d_\Sigma(\phi(y), \bdy \Omega) \geq \quasi + \stab + 2.
\end{align*}
Let $a$ and $b$ be the intersections of $[k, \ell]$ with $\bdy
\calB(z, \radius)$, chosen so that $[k, a]$ and $[b, \ell]$ meet
$\calB(z, \radius)$ at the vertices $a$ and $b$ only.  Connect $a$ to
$b$ via a path $P$ of length $\error$, outside of $\calB(z, \radius -
1)$, as provided by \reflem{FinitelyMany}.

Let $\alpha = \phi(a)$ and $\beta = \phi(b)$.  Now, any consecutive
vertices of $P$ are mapped by $\phi$ to vertices of $\calC(\Sigma)$
that are at distance at most $2\quasi$.  Connecting these by geodesic
segments gives a path $\Pi$ from $\alpha$ to $\beta$.

Note that $\Pi$ has length at most $2\quasi \error$.  Since every
vertex of $\phi(P)$ is $(\quasi + \stab + 2)$--far from $\bdy \Omega$
every vertex of $\Pi$ is $(\stab + 2)$--far from $\bdy \Omega$.  So
every vertex of $\Pi$ cuts $\Omega$.  It follows that
$d_\Omega(\alpha, \beta) \leq 4\quasi \error$, by
\reflem{LipschitzProjection}.

All that remains is to bound $d_\Omega(\kappa, \alpha)$ and
$d_\Omega(\beta, \lambda)$.  It suffices, by the Bounded Geodesic
Image Theorem, to show that every vertex of $[\kappa, \alpha]$ cuts
$\Omega$.  The same will hold for $[\beta, \lambda]$.

Every vertex of $[\kappa, \alpha]$ is $\stab$--close to a vertex of
$\phi([k, a])$.  But each of these is $(\quasi + \stab + 2)$--far from
$\bdy \Omega$.  This completes the proof.
\end{proof}

\section{The induced map on markings}
In this section, given a quasi-isometric embedding of one curve
complex into another we construct a coarsely Lipschitz map between the
associated marking complexes.

Let $\calM(S)$ and $\calM(\Sigma)$ be the marking complexes of $S$ and
$\Sigma$ respectively. Let $p \from \calM(S) \to \calC(S)$ and $\pi
\from \calM(\Sigma) \to \calC(\Sigma)$ be maps that send a marking to
some curve in that marking.

\begin{theorem}
\label{Thm:LipschitzExtension}
Suppose that $\xi(S) \geq 2$ and $\phi \from \calC(S) \to
\calC(\Sigma)$ is a $\quasi$--quasi-isometric embedding.  Then $\phi$
induces a coarse Lipschitz map $\Phi \from \calM(S) \to \calM(\Sigma)$
so that the diagram
\[
\begin{CD} 
\calM(S) @>\Phi>> \calM(\Sigma)\\ 
 @VVpV               @VV\pi V  \\ 
\calC(S) @>\phi>> \calC(\Sigma) 
\end{CD} 
\]
commutes up to an additive error. Furthermore, if $\phi$ is a
quasi-isometry then so is $\Phi$.
\end{theorem}

\begin{proof}
For a marking $m$ and laminations $k$ and $\ell$, we say the triple
$(m, k, \ell)$ is $\cobound$--admissible if
\begin{itemize}
\item $d_S \big(m, [k, \ell] \big) \leq 3$ and
\item the pairs $(k, m)$, $(m, \ell)$ and $(k,\ell)$ are $\cobound$--cobounded.
\end{itemize}
For $\cobound$ large enough and for every marking $m$,
\refprop{Extension} shows that there exists a $\cobound$-admissible
triple $(m, k, \ell)$.

Given a $\cobound$-admissible triple $(m, k, \ell)$ we will now
construct a triple $(\mu, \kappa, \lambda)$ for $\Sigma$. Let $\alpha$
be any curve in $\phi(m) \subset \calC(\Sigma)$, $\kappa = \phi(k)$
and $\lambda = \phi(\ell)$.  Note that
\begin{equation}
d_\Sigma \big( \alpha, [\kappa, \lambda] \big) \leq 4 \quasi + \stab_\Sigma,
\end{equation}
by the stability of quasi-geodesics (\reflem{Stability}).  Also
$(\kappa, \lambda)$ is a $\fancy(\cobound)$--cobounded pair, by
\refthm{Cobounded}. Let $\beta$ be a closest point projection of
$\alpha$ to the geodesic $[\kappa, \lambda]$.  By \reflem{Tight}, the
pair $(\beta, \kappa)$ is $(\fancy(\cobound) + \tight)$--cobounded.
Using \reflem{Completion}, there is a marking $\mu$ so that $\beta \in
\base(\mu)$ and $(\mu, \kappa)$ are $(\fancy(\cobound)+\tight +
\completion)$--cobounded.  Therefore, for $\Cobound =
2\fancy(\cobound) + \tight + \completion$ the triple $(\mu, \kappa,
\lambda)$ is $\Cobound$--admissible. Define $\Phi(m)$ to be equal to
$\mu$.

\begin{figure}[htbp]
\labellist
\small\hair 2pt
\pinlabel {$\calC(\Sigma)$} [Bl] at 139.3 163.4
\pinlabel {$\lambda$} [Br] at 13.3 134.8
\pinlabel {$\lambda'$} [tr] at 15.3 39.4
\pinlabel {$\kappa$} [Bl] at 173.1 118.2
\pinlabel {$\kappa'$} [tl] at 168.4 50.9
\pinlabel {$\alpha$} [l] at 96.9 115.9 
\pinlabel {$\alpha'$} [l] at 79.8 50.4 
\pinlabel {$\mu$} [bl] at 96.6 92.6
\pinlabel {$\mu'$} [tl] at 79.8 75.3
\endlabellist
\[
\begin{array}{c}
\includegraphics[scale=1]{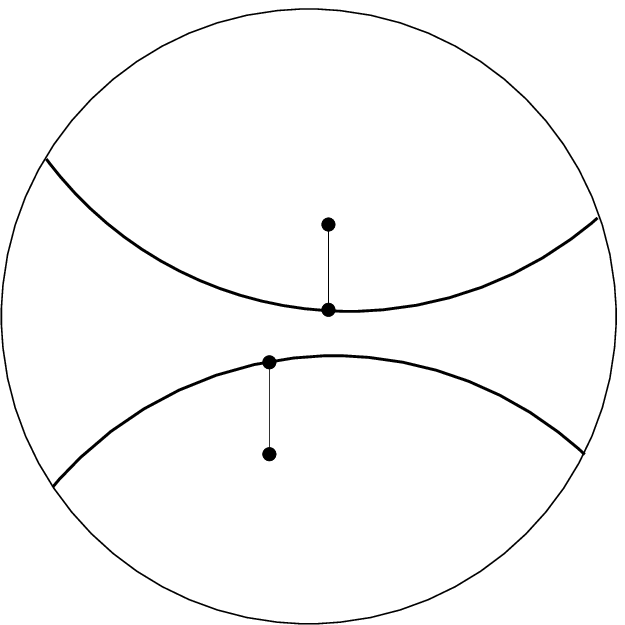}
\end{array}
\]
\caption{Markings $\mu$ and $\mu'$ are bounded apart.}
\label{Fig:Image}
\end{figure}

We now prove $\Phi$ is coarsely well-defined and coarsely Lipschitz.
Suppose that $m$ and $m'$ differ by at most one elementary move and
the triples $(m, k, \ell)$ and $(m', k', \ell')$ are
$\cobound$-admissible. Let $(\mu, \kappa, \lambda)$ and $(\mu',
\kappa', \lambda')$ be any corresponding $\Cobound$--admissible
triples in $\Sigma$, as constructed above. (See \reffig{Image}.)  We
must show that there is a uniform bound on the distance between $\mu$
and $\mu'$ in the marking graph.  By \reflem{ElementaryII}, it
suffices to prove:

\begin{claim*}
For every subsurface $\Omega \subseteq \Sigma$, $d_\Omega(\mu, \mu') =
O(1)$.
\end{claim*}

\noindent
Now, \reflem{ElementaryI} gives $d_S(m, m') \leq
4$. Deduce
$$
d_\Sigma(\phi(m), \phi(m')) \leq 5\quasi .
$$ 
Therefore,
\begin{align*}
d_\Sigma(\mu, \mu') 
 & \leq d_\Sigma \big( \mu, \phi(m) \big) 
   + d_\Sigma \big(\phi(m), \phi(m') \big) 
   + d_\Sigma \big(\phi(m'), \mu') \big) \\
 & \leq  2 (7 \quasi + \stab_\Sigma + 2) + 5 \quasi.
\end{align*}
On the other hand, for any strict subsurface $\Omega \subset \Sigma$, we have
$$
d_\Omega(\mu, \mu') \leq d_\Omega(\mu, \kappa) + 
   d_\Omega(\kappa, \kappa') + d_\Omega(\kappa', \mu'). 
$$
The first and third terms on the right are bounded by $\Cobound$.
By \refthm{Cobounded}, the second term is bounded by 
$\fancy(2 \cobound+4)$. This is because, for every strict subsurface 
$Y \subset S$, 
$$
d_Y(k, k') \leq d_Y(k, m) + d_Y(m, m') + d_Y(m', k') \leq 2\cobound + 4.
$$ 
This proves the claim; thus $\Phi$ is coarsely well-defined and
coarsely Lipschitz.

Now assume that $\phi$ is a quasi-isometry with inverse $f \from
\calC(\Sigma) \to \calC(S)$.  Let $\Phi$ and $F$ be the associated
maps between marking complexes.  We must show that $F \circ \Phi$ is
close to the identity map on $\calC(S)$.  Fix $m \in \calM(S)$ and
define $\mu = \Phi(m)$, $m' = F(\mu)$.  For any admissible triple
$(m, k, \ell)$ we have, as above, an admissible triple $(\mu, \kappa,
\ell)$.  Since $f(\kappa) = k$ and $f(\lambda) = \ell$ it follows that
$(m', k, \ell)$ is also admissible for a somewhat larger constant.

As above, by \reflem{ElementaryII} it suffices to show that $d_Z(m,
m') = O(1)$ for every essential subsurface $Z \subset S$.  Since $f
\circ \phi$ is close to the identity, commutivity up to additive error
implies that $d_S(m, m')$ is bounded.  Since $(m, k)$ and $(m', k)$
are cobounded, the triangle inequality implies that $d_Z(m, m') =
O(1)$ for strict subsurfaces $Z \subset S$.  This completes the proof.
\end{proof}

\section{Rigidity of the curve complex}

\begin{theorem}
\label{Thm:Rigid}
Suppose that $\xi(S) \geq 2$.  Then every quasi-isometry of $\calC(S)$
is bounded distance from a simplicial automorphism of $\calC(S)$.
\end{theorem}

\begin{proof}
Let $f \from \calC(S) \to \calC(S)$ be a $\quasi$-quasi-isometry.  By
\refthm{LipschitzExtension} there is a $\Quasi$-quasi-isometry $F
\from \calM(S) \to \calM(S)$ associated to $f$.  By \refthm{MCGRigid}
the action of $F$ is uniformly close to the induced action of some
homeomorphism $G \from S \to S$.  That is,
\begin{equation} 
\label{Eqn:FG}
d_\calM \big(F(m), G(m) \big) =O(1).
\end{equation}
Let $g \from \calC(S) \to \calC(S)$ be the simplicial automorphism induced 
by $G$.  We need to show that $f$ and $g$ are equal in $\QI(\calC(S))$.
Fix a curve $a \in \calC(S)$. We must show the distance $d_S(f(a), g(a))$ 
is bounded by a constant independent of the curve $a$.
Choose a marking $m$ containing $a$ as a base curve. Note
that $d_S(a, p(m)) \leq 2$, thus
$$
d_S \big( f(a), f(p(m)) \big) \leq 3 \quasi.
$$
By \refthm{LipschitzExtension}, for every marking $m \in \calM(S)$,
\begin{align*}
\label{Eqn:Additive}
d_S \big( f(p(m)), p(F(m)) \big) = O(1).
\end{align*} 
From \refeqn{FG} and \reflem{ElementaryI} we have
$$
d_S \big( p(F(m)), p(G(m)) \big) = O(1).
$$
Also, $g(a)$ is a base curve of $G(m)$, hence
$$
d_S\big( p\big( G(m) \big), g(a) \big) \leq 2. 
$$
These four equations imply that
$$
d_S \big( f(a), g(a) \big) = O(1).
$$
This finishes the proof.
\end{proof}

\appendix
\section{Classifying mapping class groups}
\label{App:Ivanov}


For any compact, connected, orientable surface $S$ let $\MCG(S)$ be
the extended mapping class group: the group of homeomorphisms of $S$,
considered up to isotopy.  We will use Ivanov's characterization of
Dehn twists via {\em algebraic twist subgroups}~\cite{Ivanov88}, the
action of $\MCG(S)$ on $\PML(S)$, and the concept of a {\em
bracelet}~\cite{BehrstockMargalit06} to give a detailed proof of:

\begin{theorem}
\label{Thm:Ivanov}
Suppose that $S$ and $\Sigma$ are compact, connected, orientable
surfaces with $\MCG(S)$ isomorphic to $\MCG(\Sigma)$.  Then either
\begin{itemize}
\item
$S$ and $\Sigma$ are homeomorphic, 
\item
$\{ S, \Sigma \} = \{ S_1, S_{1,1} \}$, or
\item
$\{ S, \Sigma \} = \{ \Sph, \DD \}$.
\end{itemize}
\end{theorem}

By the classification of surfaces, $S$ is determined up to
homeomorphism by the two numbers $\g = \genus(S)$ and $\boundary =
|\bdy S|$.  So to prove \refthm{Ivanov} it suffices prove that $\xi(S)
= 3\g -3 + \boundary$, the complexity of $S$, and $\g$, the genus, are
{\em algebraic}: determined by the isomorphism type of $\MCG(S)$.

\begin{remark}
\label{Rem:Ivanov}
\refthm{Ivanov} is a well-known folk-theorem.  A version of
\refthm{Ivanov}, for pure mapping class groups, is implicitly
contained in~\cite{Ivanov88} and was known to N.~Ivanov as early as
the fall of 1983~\cite{Ivanov07}.  Additionally, Ivanov and McCarthy
\cite{IvanovMcCarthy99} prove \refthm{Ivanov} when $\g \geq 1$ (see
also \cite{IvanovMcCarthy95}).

There is also a ``folk proof'' of \refthm{Ivanov} relying on the fact
that the rank and virtual cohomological dimension are algebraic and
give two independent linear equations in the unknowns $\g$ and
$\boundary$.  However, the formula for the vcd changes when $\g = 0$
and when $\boundary = 0$.  Thus there are two infinite families of
pairs of surfaces which are not distinguished by these invariants.

These difficult pairs can be differentiated by carefully considering
torsion elements in the associated mapping class groups.  We prefer
the somewhat lighter proof of \refthm{Ivanov} given here.
\end{remark}

The {\em rank} of a group is the size of a minimal generating set. The
{\em algebraic rank} of a group $G$, $\rank(G)$, is the maximum of the
ranks of free abelian subgroups $H < G$.  Now, the algebraic rank of
$\MCG(S)$ is equal to $\xi(S)$, when $\xi(S) \geq 1$
(Birman-Lubotzky-McCarthy~\cite{BirmanEtAl83}).  When $\xi(S) \leq 0$
the algebraic rank is zero or one.

So when $\xi(S) \geq 1$ the complexity is algebraically determined.
There are only finitely many surfaces having $\xi(S) < 1$; we now
dispose of these and a few other special cases.

\subsection*{Low complexity}

A Dehn twist along an essential, non-peripheral curve in an orientable
surface has infinite order in the mapping class group.  
Thus, the only surfaces where the mapping class group has algebraic
rank zero are the sphere, disk, annulus, and pants.  We may compute
these and other low complexity mapping class groups using the {\em
Alexander method}~\cite{FarbMargalit10}.

For the sphere and the disk we find
$$
\MCG (\Sph),~\MCG(\DD) \isom \ZZ_2,
$$ where $\ZZ_2$ is the group of order two generated by a reflection.

For the annulus and pants we find
$$
\MCG(\Ann) \isom K_4 
\quad\text{and}\quad \MCG(S_{0,3}) \isom 
\ZZ_2 \times \Sigma_3,
$$ 
where $K_4$ is the Klein $4$--group and $\Sigma_3$ is the symmetric
group acting on the boundary of $S_{0,3}$.  Here, in addition to
reflections, there is the permutation action of $\MCG(S)$ on $\bdy S$.


The surfaces with algebraic rank equal to one are the torus,
once-holed torus, and four-holed sphere.  For the torus and the
once-holed torus we find
$$
\MCG (\TT),~\MCG(S_{1,1}) \isom \GL(2, \ZZ).
$$ 
The first isomorphism is classical~\cite[Section 6.4]{Stillwell93}.
The second has a similar proof: the pair of curves meeting once is
replaced by a pair of disjoint arcs cutting $S_{1,1}$ into a disk.


Now we compute the mapping class group of the four-holed sphere:
$$
\MCG (S_{0,4})  \isom  K_4 \rtimes \PGL(2, \ZZ).
$$ 
The isomorphism arises from the surjective action of $\MCG(S_{0,4})$
on the Farey graph.  If $\phi$ lies in the kernel then $\phi$ fixes
each of the slopes $\{ 0, 1, \infty\}$ setwise.  Examining the induced
action of $\phi$ on these slopes and their intersections shows that
$\phi$ is either the identity or one of the three ``fake''
hyperelliptic involutions.  It follows that $\MCG(S_{0,4})$ has no
center,
and so distinguishes $S_{0,4}$ from $S_1$ and $S_{1,1}$.


All other compact, connected, orientable surfaces have algebraic rank
equal to their complexity and greater than one (again,
see~\cite{BirmanEtAl83}).  Among these $S_{1,2}$ and $S_2$ are the
only ones with mapping class group having nontrivial
center~\cite{FarbMargalit10}.
As the complexities of $S_{1,2}$ and $S_2$ differ, their mapping class
groups distinguish them from each other and from surfaces with equal
complexity.  This disposes of all surfaces of complexity at most three
{\em except} for telling $S_{0,6}$ apart from $S_{1,3}$.  We defer
this delicate point to the end of the appendix.

\subsection*{Characterizing twists}

Recall that if $a \subset S$ is a essential non-peripheral simple
closed curve then $T_a$ is the {\em Dehn twist} about $a$.  We call
$a$ the {\em support} of the Dehn twist.  If $a$ is a separating curve
and one component of $S \setminus a$ is a pair of pants then $a$ is a
{\em pants curve}.  In this case there is a {\em half-twist},
$T_a^{1/2}$, about $a$.  (When $a$ cuts off a pants on both sides then
the two possible half-twists differ by a fake hyperelliptic.)

We say that two elements $f, g \in \MCG(S)$ {\em braid} if $f$ and $g$
are conjugate and satisfy $fgf = gfg$. From~\cite[Lemma
4.3]{McCarthy86} and~\cite[Theorem 3.15]{IvanovMcCarthy99} we have:

\begin{lemma}
\label{Lem:Twist}
Two powers of twists commute if and only if their supporting curves
are disjoint.  Two twists braid if and only if their supporting curves
meet exactly once. \qed
\end{lemma}

The proof generalizes to half-twists along pants curves:

\begin{lemma}
\label{Lem:HalfTwist}
Two half-twists commute if and only if the underlying curves are
disjoint.  Two half-twists braid if and only if the underlying curves
meet exactly twice. \qed
\end{lemma}

Now, closely following Ivanov~\cite[Section 2]{Ivanov88}, we essay an
algebraic characterization of twists on nonseparating curves inside of
$\MCG(S)$.
Define a subgroup $H < \MCG(S)$ to be an {\em algebraic twist
subgroup} if it has the following properties.
\begin{itemize}
\item
$H = \langle g_1, \ldots, g_k \rangle$  is a free abelian group of 
rank $k = \xi(S)$, 
\item 
for all $i, j$ the generators $g_i, g_j$ are conjugate inside of $\MCG(S)$,
\item
for all $i$ and $n$ the center of the centralizer, $Z(C(g_i^n))$, is
cyclic, and
\item
for all $i$, the generator $g_i$ is not a proper power in $C(H)$.
\end{itemize}

\noindent
We have:

\begin{theorem}
\label{Thm:CharTwist}
Suppose that $\xi(S) \geq 2$ and $\MCG(S)$ has trivial center.
Suppose that $H = \langle g_1, \ldots, g_k \rangle$ is an algebraic
twist group.  Then the elements $g_i$ are all either twists on
nonseparating curves or are all half-twists on pants curves.
Furthermore, the underlying curves for the $g_i$ form a pants
decomposition of $S$.
\end{theorem}

\begin{proof}
By work of Birman-Lubotzky-McCarthy~\cite{BirmanEtAl83} (see
also~\cite{McCarthy86}) we know that there is a power $m$ so that each
element $f_i = g_i^m$ is either a power of Dehn twist or a
pseudo-Anosov supported in a subsurface of complexity one.  Suppose
that $f_i$ is pseudo-Anosov with support in $Y \subset S$; then the
center of the centralizer $Z(C(f_i))$ contains the group generated by
$f_i$ and all twists along curves in $\bdy Y$.  However, this group is
not cyclic, a contradiction. 

Deduce instead that each $f_i$ is the power of a Dehn twist.  Let
$b_i$ be the support of $f_i$.  Since the $f_i$ commute with each
other and are not equal it follows that the $b_i$ are disjoint and are
not isotopic.  Thus the $b_i$ form a pants decomposition.  Since the
$g_i$ are conjugate the same holds for the $f_i$.  Thus all the $b_i$
have the same topological type.

If $S$ has positive genus then it follows that all of the $b_i$ are
nonseparating.  If $S$ is planar then it follows that $S = S_{0,5}$ or
$S_{0,6}$ and all of the $b_i$ are pants curves.

Now fix attention on any $h \in C(H)$, the centralizer of $H$ in
$\MCG(S)$.  We will show that $h$ preserves each curve $b_i$.  First
recall that, for any index $i$ and any $n \in \ZZ$, the element $h$
commutes with $f_i^n$.  Let $\ell \in\PMF(S)$ be any filling
lamination.  We have
\begin{gather*}
f_i^n(\ell) \to b_i \quad\text{as}\quad n \to \infty,\\
\tag*{so}
h \circ f_i^n(\ell) \to h(b_i) \quad\text{and}\quad
f_i^n \circ h (\ell) \to b_i.
\end{gather*}
It follows that $h(b_i) = b_i$ for all $i$.  

Fix attention on any $b_\ell \in \{ b_i \}$.  Suppose first that the
two sides of $b_\ell$ lie in a single pair of pants, $P$.  If $\bdy P$
meets $b_\ell$ and no other pants curve then $S = S_{1,1}$, a
contradiction.  If $\bdy P$ meets only $b_\ell$ and $b_k$ then
$b_\ell$ is nonseparating and $b_k$ is separating, a contradiction.
We deduce that $b_\ell$ meets two pants, $P$ and $P'$.  Now, if $h$
interchanges $P$ and $P'$ then $S$ is in fact the union of $P$ and
$P'$; as $\xi(S) > 1$ it follows that $S = S_{1,2}$ or $S_2$.
However, in both cases the mapping class group has non-trivial center,
contrary to hypothesis.

We next consider the possibility that that $h$ fixes $P$ setwise.  So
$h|P$ is an element of $\MCG(P)$.  If $h|P$ is orientation reversing
then so is $h$; thus $h$ conjugates $f_\ell$ to $f_\ell^{-1}$, a
contradiction.  If $h|P$ permutes the components of $\bdy P \setminus
b_\ell$ then either $b_\ell$ cuts off a copy of $S_{0, 3}$ or
$S_{1,1}$ from $S$; in the latter case we have $b_i$'s of differing
types, a contradiction.  

To summarize: $h$ is orientation preserving, after an isotopy $h$
preserves each of the $b_i$, and $h$ preserves every component of $S
\setminus \{ b_i \}$.  Furthermore, when restricted to any such
component $P$, the element $h$ is either isotopic to the identity or
to a half twist.  The latter occurs only when $P \cap \bdy S =
\delta_+ \cup \delta_-$, with $h(\delta_\pm) = \delta_\mp$.

So, if the $b_i$ are nonseparating then $h$ is isotopic to the
identity map on $S \setminus \{ b_i \}$.  It follow that $h$ is a
product of Dehn twists on the $b_i$.  In particular this holds for
each of the $g_i$ and we deduce that $T_{b_i}$, the Dehn twist on
$b_i$, is an element of $C(H)$.  Now, since $Z(C(g_i)) = \ZZ$ we
deduce that the support of $g_i$ is a single curve.  By the above
$g_i$ is a power of $T_{b_i}$; since $g_i$ is primitive in $C(H)$ we
find $g_i = T_{b_i}$, as desired.

The other possibility is that the $b_i$ are all pants curves.  It
follows that $S = S_{0,5}$ or $S_{0,6}$.  Here $h$ is the identity on
the unique pants component of $S \setminus \{ b_i \}$ meeting fewer
than two components of $\bdy S$.  On the others, $h$ is either the
identity or of order two.  So $h$ is a product of half-twists on the
$\{ b_i \}$.  As in the previous paragraph, this implies that $g_i =
T_a^{1/2}$.
\end{proof}

\subsection*{Bracelets}
We now recall a pretty definition from~\cite{BehrstockMargalit06}.
Suppose $g$ is a twist or a half-twist.  A {\em bracelet} around $g$
is a set of mapping classes $\{ f_i \}$ so that
\begin{itemize}
\item
every $f_i$ braids with $g$ (and so is conjugate to $g$),
\item
if $i \neq j$ then $f_i \neq f_j$ and $[f_i, f_j]=1$, and 
\item
no $f_i$ is equal to $g$.
\end{itemize}

\noindent
Note that bracelets are algebraically defined.  The \emph{bracelet
number} of $g$ is the maximal size of a bracelet around $g$.

\begin{claim}
A half-twist on a pants curve has bracelet number at most two. 
\end{claim}

\begin{proof}
If two half-twists braid, then they intersect twice
(\reflem{HalfTwist}).  Thus, the pants they cut off share a curve of
$\bdy S$. If two half-twists commute, the pants are disjoint (again,
\reflem{HalfTwist}). Therefore, there are at most two commuting
half-twists braiding with a given half-twist.
\end{proof}

On the other hand: 

\begin{claim}
Suppose that $S \neq \TT$.  Then a twist on a nonseparating curve has
bracelet number $2\g - 2 + \boundary$.
\end{claim}

\begin{proof}
Suppose that $a$ is a non-separating curve with associated Dehn twist.
Let $\{ b_i \}$ be curves underlying the twists in the bracelet.  Each
$b_i$ meets $a$ exactly once (\reflem{Twist}).  Also, each $b_i$ is
disjoint from the others and not parallel to any of the others (again,
\reflem{Twist}).  Thus the $b_i$ cut $S$ into a collection of surfaces
$\{ X_j \}$.  For each $j$ let $a_j = a \cap X_j$ be the remains of
$a$ in $X_j$.  Note that $a_j \neq \emptyset$.  Each component of
$\bdy X_j$ either meets exactly one endpoint of exactly one arc of
$a_j$ or is a boundary component of $S$.  Now:
\begin{itemize}
\item
if $X_j$ has genus,
\item
if $|\bdy X_j| > 4$, or
\item
if $|\bdy X_j| = 4$ and $a_j$ is a single arc,
\end{itemize}
then there is a non-peripherial curve in $X_j$ meeting $a_j$
transversely in a single point.  However, this contradicts the
maximality of $\{ b_i \}$.  Thus every $X_j$ is planar and has at most
four boundary components.  If $X_j$ is an annulus then $S$ is a torus,
violating our assumption.  If $X_j $ is a pants then $a_j$ is a single
arc.  If $X_j = S_{0,4}$ then $a_j$ is a pair of arcs.  An Euler
characteristic computation finishes the proof.
\end{proof}

\subsection*{Proving the theorem}

Suppose now that $\xi(S) \geq 4$.  By \refthm{CharTwist} any basis
element of any algebraic twist group is a Dehn twist on a
nonseparating curve.  Here the bracelet number is $2\g - 2 +
\boundary$.  Thus $\xi(S)$ minus the bracelet number is $\g - 1$.
This, together with the fact that $\xi(S)$ agrees with the algebraic
rank, gives an algebraic characterization of $\g$.

When $\xi(S) \leq 2$ the discussion of low complexity surfaces proves
the theorem.  The same is true when $\xi(S) = 3$ and $\MCG(S)$ has
nontrivial center.

The only surfaces remaining are $S_{0,6}$ and $S_{1,3}$.  In
$\MCG(S_{0,6})$ every basis element of every algebraic twist group has
bracelet number two.  In $\MCG(S_{1,3})$ every basis element of every
algebraic twist group has bracelet number three.  So \refthm{Ivanov}
is proved. \qed

\bibliographystyle{plain}
\bibliography{bibfile}
\end{document}